\documentclass[12pt]{amsart}

\usepackage{amstext}
\usepackage{amsthm}
\usepackage{amsmath}
\usepackage{latexsym}
\usepackage{amsfonts}
\usepackage{mathtools}
\usepackage{enumerate}
\usepackage{geometry}
\usepackage{blindtext}
\usepackage{incgraph,tikz}
\usepackage[utf8]{inputenc}
\usepackage{bbold}
\usepackage{nameref}
\usepackage[normalem]{ulem}
\usepackage{wasysym}
\usepackage{csquotes}
\usepackage{anyfontsize}
\usepackage{hyperref}

\newtheorem*{problem*}{Open Problem}
\newtheorem{theorem}{Theorem}[section]
\newtheorem*{question*}{Question}
\newtheorem{lemma}[theorem]{Lemma}

\newtheorem{problem}[theorem]{Problem}
\theoremstyle{definition}
\theoremstyle{definition}
\theoremstyle{definition}
\numberwithin{equation}{section}

\allowdisplaybreaks

\newcommand\Hol{{\mathrm{Hol}(\mathbb D)}}

\newcommand\CC{{\mathbb C}}
\newcommand\RR{{\mathbb R}}
\newcommand\DD{{\mathbb D}}
\newcommand\TT{{\mathbb T}}

\newcommand\BB{{\mathbb B_n}}
\newcommand\SSS{{\mathbb S_n}}
\newcommand\dBB{{\overline{\mathbb B}_n}}

\title[cyclicity and polynomials]{An overview of cyclicity in Dirichlet-type spaces in the unit ball}
\author{Dimitrios Vavitsas}
\email{vavitsas@mail.sysu.edu.cn}
\address{School of Mathematics (Zhuhai), Sun Yat-Sen University, Zhuhai, Guangdong, 519082, P. R.
China}
\thanks{}

\begin{document}
\begin{abstract}
We shall discuss the characterization of cyclicity of polynomials in Dirichlet-type spaces in the Euclidean unit ball of $\CC^n$ which remains an open problem.
\end{abstract}

\maketitle

\section{Background}

For $z=(z_1,\dots,z_n)$ and $w=(w_1,\dots,w_n)\in\CC^n,$ the \textit{Euclidean inner product} $\langle z,w\rangle=z_1\overline{w}_1+\dots+z_n\overline{w}_n$ and its associated \textit{Euclidean} norm $||z||=\sqrt{|z_1|^2+\dots+|z_n|^2}$ make $\CC^n$ into a $n$-dimensional Hilbert space whose the open \textit{unit ball} will be denoted by $\BB=\{z\in\CC^n:||z||<1\}.$ 

The unit ball is a complete Reinhardt domain containing the origin and its topological boundary, the \textit{unit sphere} $\SSS=\{z\in\CC^n:||z||=1\}$, coincides with its Shilov boundary. Roughly speaking the Shilov boundary of a bounded domain in $\CC^n$ is the smallest subset of its topological boundary where an analogue of the maximum modulus principle take place and it is not always true that it coincides with the whole topological boundary. For instance, let us denote the \textit{unit polydisk} by $\DD_n:=\{z\in \CC^n: |z_1|,...,|z_n|<1\};$ the Cartesian product of $n$-unit disks. Its Shilov boundary is the \textit{unit Torus} $\TT^n:=\{z\in \CC^n:|z_1|=...=|z_n|=1\}$ which is much smaller than the topological boundary $\partial \DD_n=\overline{\DD}_n\setminus\DD_n.$ See \cite{Rein} and \cite{Stout} for more information on Reinhardt domains and Shilov boundaries, respectively. 

Denote the space of holomorphic functions in $\BB$ by $\textup{Hol}(\BB)$ and let $f\in \Hol.$ The geometry of the unit ball allows the function $f$ to have a power series expansion of the form:
$$f(z)=\sum\limits_{|k|=0}^{+\infty}a_kz^k, \quad\quad z=(z_1,\dots,z_n)\in\BB,$$
where $k=(k_1,\dots,k_n)$ is a multi-index of non-negative integers, $|k|=k_1+\dots+k_n$ and $z^k=z_1^{k_1}\cdots z_n^{k_n}$. The power series converges normally in $\BB,$ namely, $\sum_{k}\sup_{K}|a_kz^k|$ converges for each compact $K\subset \BB$, and this expansion of $f$ with respect to the origin is unique. See \cite{Horm} for a background on holomorphic functions in Reinhardt domains. 

\subsection{Dirichlet-type spaces}
Fix a real parameter $\alpha\in\RR.$ We say that $f$ belongs to the \textit{Dirichlet-type space} $D_\alpha(\BB)$ if
$$||f||^2_\alpha:=\sum\limits_{|k|=0}^{+\infty}(n+|k|)^\alpha\frac{(n-1)!k!}{(n-1+|k|)!}|a_k|^2<+\infty,$$
where $k!=k_1!\cdots k_n!$. The family of Dirichlet-type spaces depends on a real parameter and consists of Hilbert spaces of holomorphic functions. The $\alpha$-Dirichlet-type norm depends on the power series expansion of a holomorphic function and a weight that depends on the real parameter $\alpha.$ 

Special cases of this family are some of the classical Hilbert spaces of holomorphic functions in the unit ball of $\CC^n$. For instance, the following cases play a significant role in complex Analysis and operator Theory: the spaces  $D_{-1}(\BB), D_0(\BB), D_{n-1}(\BB)$ and $D_n(\BB)$ coincide with the well known Bergman, Hardy, Drury-Arveson and Dirichlet spaces, respectively. For more information about these spaces, we refer the interested reader to \cite{Ahern}, \cite{Perfekt}, \cite{Beatrous}, \cite{Feng}, \cite{Hartz}, \cite{Hu shi}, \cite{Hu},  \cite{Vavitsas1}, \cite{Li}, \cite{Michalska},  \cite{Vavitsas2}, \cite{Zhu}. See also \cite{Sampat}  for operators that preserve shift-cyclic functions and weighted composition operators, \cite{Puri2} about defining this family of spaces in more general domains, and \cite{Datta} for a recent discussion on cyclicity via weak$^*$ sequentially cyclicity in spaces of holomorphic functions.

Fundamental properties of Dirichlet-type spaces are the density of the space of polynomials in each space $D_\alpha(\BB)$, $\alpha\in \RR$; for each $f\in D_\alpha(\BB),$ the function $z_i\cdot f \in D_\alpha(\BB),$ $i = 1, ..., n;$  the  linear \textit{point evaluation functional} with respect to any fixed point $z\in \BB:$ $x^*_z:D_\alpha(\BB)\rightarrow \CC,$ defined by $x^*_z(f)(z):=f(z),$ is a continuous linear functional.

A key concept of differentiation on the unit ball setting is that of the \textit{radial derivative}:
$$R(f)(z):=z_1\partial_{z_1}f(z)+...+z_n\partial_{z_n}f(z), \quad z\in \BB,$$
where $f\in \Hol.$ Furthermore, it easy to verify the following relation among Dirichlet-type spaces:
\begin{equation}\label{among}
   f\in D_\alpha(\BB) \quad \text{ iff } \quad R(f)\in D_{\alpha-2}(\BB). 
\end{equation}
One can use  this relationship in order to solve estimation problems in a more convenient space. However, one cannot always easily work with the norm due to the fact that it is defined based on the Taylor coefficients. It is crucial in many cases where the norm needs to be estimated or computed, to find equivalent integral norms in order to use analytic methods. Let us define the \textit{gradient} of a holomorphic function $f\in \Hol:$
$$ \nabla(f)(z)=(\partial_{z_1}f(z),...,\partial_{z_n}f(z)), \quad z\in \BB.$$
The following norm was crucial in the characterization of cyclicity of polynomials in the two dimensional ball setting, see \cite{Vavitsas1}.
\begin{lemma}\label{eq norm}
Let $\alpha\in (-1,1)$ be fixed and define $\mathcal{N}_\alpha:D_\alpha(\BB)\rightarrow [0,\infty)$ by 
$$\mathcal{N}_\alpha(f)=\sqrt{|f(0)|^2+|f|_\alpha^2},$$
where
$$|f|^2_\alpha:=\int_{\BB} \frac{||\nabla(f)(z)||^2 - |R(f)(z)|^2}{ (1-||z||^2)^\alpha} dA(z).$$
Then $\mathcal{N}_\alpha$ is a norm in $D_\alpha(\BB)$ equivalent to $||\cdot||_{D_\alpha(\BB)}.$
\begin{proof}
This is an immediate consequence of the arguments in \cite{Michalska} and the expansion $(3)$ in p. 349 of \cite{Vavitsas1}.
\end{proof}
\end{lemma}	
Let us mention that the equivalent characterization $f\in D_\alpha(\BB)$ if and only if $|f|_\alpha<+\infty,$ was first claimed by S. Li in \cite{Li}. However, the proof relied on a technical lemma concerning the computation of certain series, which turns out to be false. The error was later detected by M. Michalska, who established the equivalence for the range $\alpha\in (-1,1),$ see \cite{Michalska}.

\subsection{Cyclic vectors}
Consider the so called \textit{shift operator} $T_i:D_\alpha(\BB)\to D_\alpha(\BB),$ $i=1,..,n,$ defined by $T_i(f)(z):= z_i\cdot f(z)$. A function $f\in D_\alpha(\BB)$ is called \textit{cyclic} if the closed invariant subspace
$$[f]_\alpha:=\textrm{clos span}\{z_1^{k_1}\cdots z_n^{k_n}f:k_1,\dots,k_n=0,1,2,\dots\}$$
coincides with the whole space $D_\alpha(\BB).$ The closure is taken with respect to the norm of $D_\alpha(\BB).$ Equivalently, $f\in D_\alpha(\BB)$ is cyclic if and only if there exists a sequence of polynomials $p_n\in \CC[z]$ such that $p_nf\rightarrow 1$ in $D_\alpha$ norm or weakly. 

Let us point out that cyclic functions in $D_\alpha(\BB)$ cannot vanish in the unit ball.
This is a consequence of the continuity of the point evaluation functionals. Therefore, the study
of cyclic functions in Dirichlet-type spaces focuses on functions that are zero-free in the unit ball. A wide variety of published work points to the fact that the cyclicity of $f$ in the spaces $\mathcal{D}_\alpha,$ $\alpha\in (0,n]$ is inextricably related to the size of the \emph{zero set} $\mathcal{Z}(f^*):=\{\zeta\in \SSS:f^*(\zeta)=0\}$, where $f^*$ denotes the radial limit function of $f$, namely, $f^*(\zeta):=\lim_{r\rightarrow 1^-}f(r\zeta)$. Recall that $f^*$ is defined almost everywhere on $\SSS$ for $f\in D_\alpha(\BB),$ $\alpha\geq 0;$ this is due to the fact that the spaces are contained in the Hardy space in which the radial limit exist almost everywhere.  If $p$ is a polynomial, then $\mathcal{Z}(p^*)=\mathcal{Z}(p)\cap\SSS.$ The points lying on $\mathcal{Z}(p)\cap \SSS$ are called \emph{boundary zeroes} of $p$.

Trivial examples of cyclic functions are the non-zero constant functions. This is
a consequence of the density of the polynomials.

\section{Discussion on the characterization of cyclic polynomials}
\begin{problem}[Open]\label{conjecture}
   Let $p\in\CC[z_1,\dots,z_n]$ be a zero-free in $\BB$ polynomial. Suppose that $\mathcal{Z}(p)\cap\SSS$ contains a real submanifold of $\RR^{2n}$ of dimension $m-1$, $m=2,3,\dots,n$, but no submanifold of any higher dimension. Then $p$ is cyclic in $D_\alpha(\BB)$ if and only if $\alpha\le\frac{2n-(m-1)}{2}$.  
\end{problem}

We see in the above characterization that the cyclicity of a polynomial depends exactly on the size of the zeroes lying on the unit sphere, that is, the real dimension of the set $\mathcal{Z}(p)\cap \SSS.$ Recall that we may always regard a subset $M\subset \CC^n$ as a a subset of $\RR^{2n},$ and hence to define the \textit{real dimension}: $\textrm{dim}_{\RR}M:=\max\{\textrm{dim}\Gamma:\Gamma\subset M \text{ submanifold of } \RR^{2n}\}.$ The case $\dim_\RR M=-1$ is devoted to the instance when $M=\emptyset.$ Moreover, see \cite{Brown-Shields}, \cite{Primer} for the one-variable setting and \cite{Kosinski-Sola1}, \cite{Kosinski-Sola2} for the correspinding bidisk setting.

Let us recall the definition of a semi-algebraic set, see \cite{Algebraic} for a background. Fix $N\in\mathbb{N}$. A set $A\subset\RR^N$ is said to be \textit{semi-algebraic} if for any $x\in\mathbb{R}^N$, there exist a neighborhood $U=U(x)$ and a finite number of polynomials $f_i$, $g_{ij}$, $i=1,\dots,p$, $j=1,\dots,q$, $p,q\in\mathbb{N}$, such that
$$A\cap U=\bigcup\limits_{j=1}^{q}\bigcap\limits_{i=1}^{p}\{x\in U:f_i(x)=0, g_{ij}(x)>0\}.$$
Moreover, the set $A$ is said to be \textit{algebraic} if
$$A\cap U=\{x\in U: f_i(x)=0\}.$$

Given a polynomial $p\in\CC[z_1,\dots,z_n]$, the set $\mathcal{Z}(p)\cap\SSS$ may be viewed as an algebraic set in $\RR^{2n}$; indeed, it is the intersection of the unit sphere $\SSS=\{z\in\CC^n: 1-||z||^2=0\}$ and the two algebraic sets $\{z\in\CC^n: \textup{Im }p(z)=0\}$, $\{z\in\CC^n: \textup{Re }p(z)=0\}$. Therefore, the set of the boundary zeroes of a polynomial may be regarded as an algebraic subset of $\RR^{2n}$.

The Problem~\ref{conjecture} was formulated in terms of the explicit description of the semi-algebraic set $\mathcal{Z}(p)\cap \SSS$, see \cite{Vavitsas3}, and the characterization of cyclicity of the \textit{model polynomials} $\pi_m(z):=1-m^{\frac{m}{2}}z_1\cdots z_m,$ $m=1,...,n;$ the polynomial $\pi_m$ is cyclic in $D_\alpha(\BB)$ if and only if $\alpha\le\frac{2n-(m-1)}{2}$ and its boundary zeroes are of real dimension $\dim_{\RR}\mathcal{Z}(\pi_m)\cap \SSS=m-1,$ see \cite{Vavitsas2}. In particular, the set  $\mathcal{Z}(p)\cap \SSS$ is either a finite set or it is a finite disjoint union of Nash submanifolds - components, each Nash diffeomorphic to an open hypercube of dimension at most $n-1.$ These components are in particular real analytic submanifolds and equivalent to the hypercube by complex tangential real analytic diffeomorphisms. Each component is also a complex tangential subset of $\SSS.$ A background on semi-algebraic geometry and interpolation theory in the unit ball can be found in \cite{Algebraic} and \cite{Rudin}.

In the Problem~\ref{conjecture}, we exclude the case when $\mathcal{Z}(p)\cap \SSS$ is a finite set ($\dim_{\RR}\mathcal{Z}(p)\cap \SSS=0$) since it is possible to apply the same arguments as in \cite{Vavitsas1} or \cite{Perfekt} to overcome this case. The necessary part of cyclicity was proved in \cite{Vavitsas3}. 

It remains to prove the sufficient part, that is, if a polynomial $p\in \CC[z_1,...,z_n]$ satisfies the assumption of the statement above, then $p$ is cyclic in $D_\alpha(\BB)$ for $\alpha\le\frac{2n-(m-1)}{2}.$

\subsection{The Radial Dilation Method} 
If $f\in \Hol$ and $r\in (0,1),$ then $f_r$ denotes the \emph{dilated function} defined for $|z|<1/r$ by $f_r(z)=f(rz).$ Note that if $f$ has no zeros in $\BB,$ then $1/f_r$ is well defined on $\BB$ and $1/f_r\in \mathcal{O}(\dBB),$ where the last symbol denotes the class of the functions that are holomorphic in a neighbourhood of the closed unit ball. 

The following method for establishing cyclicity is known as the radial dilation method:
let $p\in \CC[z_1,...,z_n]$ be a polynomial satisfying the assumption of the Problem~\ref{conjecture}. Prove that
$$\limsup_{r\rightarrow 1^-}||p/p_r||_{\alpha}<\infty,$$ where $\alpha=\frac{2n-(m-1)}{2}$ and derive cyclicity.

Such tools for identifying cyclicity were crucial in \cite{Kosinski-Sola2}, \cite{Vavitsas1}, \cite{Vavitsas2} and \cite{Puri1}. For instance, let us give a short description of the steps in the proof of the cyclicity part of the main Theorem in \cite{Vavitsas1} (cyclicity of a zero-free in $\mathbb B_2$ polynomial $p\in \CC[z_1,z_2]$ in $D_\alpha(\mathbb B_2)$ for $\alpha\leq 3/2$). The main idea was to use the relation among Dirichlet-type spaces \eqref{among} in order to switch the problem of estimating $||p/p_r||_{3/2}$ to the estimation of $||R(p/p_r)||_{-1/2}.$ Thanks to Lemma~\ref{eq norm}, it is enough to estimate integrals over the unit ball. Moreover, by a compactness argument it is enough to estimate the integrals locally, that is, around each fixed point $\zeta=(\zeta_1,\zeta_2)\in \mathbb S_2.$ This allows us to write the polynomial in each Weierstrass form and deal with operations of terms $(1-z_2h_i(z_2))/(1-rz_2h_i(rz_1))$ as well as partial derivatives of them. The functions $h_i$ come out from the Weierstrass form and they are defined on slight domains around $\zeta_1,$ see \cite{Kosinski-Sola2}. Tools arising from the theory of analysis such as Puiseux's Parametrization Theorem and B\l{}ocki’s Lemma lead to the estimation of certain integrals where one may use elementary calculus techniques and identify cyclicity for $\alpha\leq 3/2.$ 

One may ask if it is possible to employ the same method in the $\BB$-setting in order to prove that $\limsup_{r\rightarrow 1^-}||p/p_r||_{\alpha}<\infty,$ where $\alpha=\frac{2n-(m-1)}{2}$. The problem in the multi-variable setting is slightly more complex, since applying this method requires additional considerations. For instance, pick $\alpha=\frac{2n-(m-1)}{2},$ $m=2,...,n.$ By \eqref{among}, there exists $q\in \mathbb{N}$ and $\upsilon\in [-1,1]$ such that $\alpha=2q+\upsilon.$ Thus,
$$f\in D_\alpha(\BB) \quad \text{ iff } \quad R^{q}(f)\in D_{\upsilon}(\BB).$$ If we focus on the case $\upsilon\in (-1,1)$, then by the equivalent norm in Lemma~\ref{eq norm}, we need to deal with the estimation of the terms $\mathcal{N}_{\upsilon}(R^q(p/p_r)).$  

We see then that the estimations involve integrals of operations on high-order partial derivatives. In order to avoid the complexity of such operations, N. Chalmoukis,  wondered whether we could find a way to shift the difficulty of operations of high order partial derivatives to one variable derivatives and simplify the terms we need to estimate. Thanks to the equalities below we are able to obtain certain integrals for estimation:
\begin{equation}\label{usef}
R(f)(wz)=w\partial_wf(wz),\quad  wz:=(wz_1,...,wz_n), \quad w\in \CC,
\end{equation}
see equation (2) in \cite{Rudin}, subsection 6.4.4.. As a consequence, if $N\in \mathbb{N},$ then
\begin{equation}\label{usef}
 R^N(f(w\cdot z))=\sum_{k=1}^{N}a_kw^k\partial^{k}_{w}f_z(w),   
\end{equation}
see the proof of  Theorem 5.6 in \cite{Perfekt}.

We may then continue as in the two-dimensional setting, that is, factorizing the polynomial by means of the Two Function Lemma-Weierstrass Parametrization Theorem, applying the equality \eqref{usef} and estimating the integrals of the norm  that arise locally in $\overline{\mathbb B}_n.$ Of course, another difficulty in the multi-variable setting is the lack of Puiseux's Parametrization Theorem.

We hope that, in the near future, we will be able to make further progress beginning from this starting point.

\subsection{Another concept for identifying cyclicity}
A concept for identifying cyclicity was put forward in \cite{Perfekt} and \cite{Puri1}. In particular, due to the explicit description of the semi-algebraic set $\mathcal{Z}(p)\cap \SSS$ (see Theorem 5, \cite{Vavitsas3}), it is natural to ask whether $\mathcal{Z}(p)\cap \SSS$ has no essential singularities or whether it has a finite covering of proper compact sets $K_j\subset \SSS$ that are peak sets for $A^\infty(\BB)$ and the corrsponding peak function are cyclic in suitable Diriclet-type spaces. If the boundary zeroes satisfies such conditions, then it might be possible to derive cyclicity applying Theorem 3.6 of \cite{Perfekt}.

We see then that the fields of semi-algebraic geometry and interpolation theory in the unit ball are highly involved in the cyclicity problem.

Let us consider the covering case. The problematic components of $\mathcal{Z}(p)\cap \SSS$ that must be covered are those having singularities in their relative boundary. Let us pick a component $M\subset \mathcal{Z}(p)\cap \SSS$ and the corresponding diffeomorphism $\Phi:M\rightarrow (-1,1)^m,$ see Theorem 5 of \cite{Vavitsas3}. Of course, the closure of $M$ is a subset of the sphere, i.e. $\overline{M}\subset \SSS$, but it must be checked whether $\overline{M}$ is a peak set for $A^\infty(\BB)$ (or whether $\overline{M}$ is a  compact $C^\infty$ manifold in $\SSS$ that is complex-tangential). See \cite{Chollet}, \cite{Fornaess} and \cite{Rudin} about peak sets. 

Recall the Wing Lemma which is a generalization of Curve Selection Lemma, see Definition 9.7.8 and Theorem 9.7.10 of \cite{Algebraic}. 

\begin{theorem}
Let $Y\subset \RR^{2n}$ be Nash submanifold and $S\subset \RR^{2n}$ a semi-algebraic set such that $Y\cap S=\emptyset$ and $Y\subset \overline{S}.$ There exists a finite number of Nash wings
$$w_i:(-1,1)\times U_i\rightarrow \RR^{2n}, \quad i=1,...,k,$$
where $U_1,...,U_k$ are open semi-algebraic subsets of $Y,$ such that $w_i((0,1)\times U_i)$ is contained in $S,$ for $i=1,...,k,$ and $\dim(Y\setminus\cup_{i=1}^{k}U_i)<\dim(Y).$
\end{theorem}

Summing up we may consider the following questions.
\begin{question*}
Is it possible to apply Wing Lemma to find a finite or a countable proper covering of $\mathcal{Z}(p)\cap \SSS$ consisting of compact peak sets for $A^\infty(\BB)?$ Consider a more straightforward case: what if we assume that $\dim_{\RR}\mathcal{Z}(p)\cap\SSS=1?$
\end{question*}

\section*{Acknowledgements}
I would like to thank N. Chalmoukis, Ł. Kosiński, K. Zarvalis, P. Ziarati for the valuable discussions we had on these topics. I would like also to thank A. Beslikas, A. Chavan and Ł. Kosiński for the invitation to participate to the conference Bench Math Session,  2025.

\bibliographystyle{plain}

\end{document}